\documentclass[12pt,oneside]{article}

\usepackage{amssymb}
\usepackage{amsmath}
\usepackage{tikz}
\usetikzlibrary{shapes,arrows,positioning}
\usepackage[letterpaper,left=1in,right=1in, bottom=1in, top=1in]{geometry}
\usepackage{titlesec}
\usepackage{setspace}
\usepackage{textcomp}
\usepackage{amsthm}
\usepackage{eso-pic}
\usepackage{graphicx}
\usepackage{color}
\usepackage{longtable}
\usepackage{transparent}
\usepackage{caption}
\usepackage{hyperref}

\renewcommand{\geq}{\geqslant}
\renewcommand{\leq}{\leqslant}

\renewcommand{\cdots}{\ldots}

\newcommand{\tab}{\hspace{1cm}}

\titleformat{\section}{\bfseries\centering}{\thesection \hspace{0.5cm}}{12pt}{}

\titleformat{\subsection}{\bfseries}{\thesubsection \hspace{0.5cm}}{12pt}{}

\newtheorem{theorem}{Theorem}[section]

\newtheorem{definition}{Definition}[section]

\begin{document}\setlength{\parindent}{0cm}\thispagestyle{empty}
	\begin{center}
		\Large \textbf{On Legendre Cordial Labeling of Some Graphs \\
		Under Graph Operations}
	\end{center}
	\vspace{2mm}
	
	\textbf{Jason D. Andoyo}
	
	University of Southeastern Philippines, Davao City, Philippines

	\section*{Abstract} \small
	\tab For a simple connected graph $G$ of order $n$, a bijective function $f:V(G)\to\{1,2,\cdots,n\}$ is said to be a Legendre cordial labeling modulo $p$, where $p$ is an odd prime, if the induced function $f_p^*:E(G)\to \{0,1\}$, defined by $f_p^*(uv)=0$ whenever $([f(u)+f(v)]/p)=-1$ or $f(u)+f(v)\equiv 0(\text{mod }p)$, and $f_p^*(uv)=1$ whenever $([f(u)+f(v)]/p)=1$, satisfies the condition $|e_{f_p^*}(0)-e_{f_p^*}(1)|\leq 1$ where $e_{f_p^*}(i)$ is the number of edges with label $i$ ($i=0,1$). This paper investigates the Legendre cordial labeling of graphs obtained through various operations: join, corona, lexicographic product, cartesian product, tensor product, and strong product.

	\vspace{2mm}
	
	\textbf{Keywords:} odd prime, Legendre cordial labeling, Legendre cordial graph
	
	\vspace{2mm}
	
	\textbf{2020 Mathematics Subject Classification:} 05C78, 11A07, 11A15
	
	\normalsize
	\section{Introduction}
	
	\tab All graphs considered in this paper are finite, simple, and connected. A graph $G=(V,E)$ has a vertex set $V=V(G)$ and an edge set $E=E(G)$, and the cardinalities $|V|$ and $|E|$ are called the order and size of $G$, respectively. In addition, the elements of $V$ and $E$ are called vertices and edges, respectively. For further graph-theoretic definitions and properties, refer to \cite{Chartrand}.
	
	\tab Mathematics often evolves from simple ideas that grow into powerful tools for understanding the world. Graph theory, for example, began with a puzzle in 1736 when Euler solved the famous Königsberg bridge problem. This laid the groundwork for representing networks using vertices and edges. Over time, graph theory became a key part of computer science, transportation planning, biology, and social networks.
	
	\tab One important concept in graph theory is \textit{graph labeling}, in which numbers or labels are assigned to the vertices, the edges, or both of a graph following specific conditions. Graph labeling has applications in communication systems, coding theory, circuit design, and frequency assignment. 
	
	\tab Rosa \cite{Rosa} introduced several important types of graph labeling, known as $\alpha$-, \-$\beta$-, $\sigma$-, and $\rho$-valuations, which laid the foundation for many forms of graph labeling studied today. One of these, the $\beta$-valuation (also known as \textit{graceful labeling}), applies to a graph with $q$ edges: labeling the $n$ vertices with distinct values from $\{0, 1, ..., q\}$ so that the edge labels, defined as the absolute differences between endpoint labels, are all distinct and range from $1$ to $q$. This becomes particularly useful in network addressing and X-ray crystallography, where unique edge labels prevent interference and overlap.
	
	\tab A specific type of graph labeling known as \textit{cordial labeling}, introduced by Cahit \cite{Cahit} in 1987, extends these ideas in a more accessible form. Inspired by Rosa’s graceful labeling, cordial labeling is considered a weaker version of graceful labeling. The idea is to label the vertices with 0s and 1s so that the resulting edge labels (often derived from vertex labels) are roughly balanced in count. Cordial labeling has given rise to many interesting variants presented in \cite{Gallian}.
	
	\tab Number theory, another profound area of mathematics, focuses on the properties of integers. Prime numbers have been studied since ancient times. Euclid proved the infinitude of primes, while Eratosthenes devised the elegant sieve method to generate primes. In modern times, prime numbers play a critical role in cybersecurity, especially through encryption schemes like RSA, which rely on the difficulty of factoring large primes. \cite{Rosen}
	
	\tab The Legendre symbol, introduced by Legendre in the late 18th century, was developed as part of his investigations into quadratic residues and reciprocity, a cornerstone topic in number theory.By definition, an integer $a$ relatively prime to an odd prime $p$ is called a quadratic residue of $p$ if the quadratic congruence $x^2\equiv a(\text{mod }p)$ has a solution; otherwise, $a$ is a quadratic nonresidue of $p$. Thus, the Legendre symbol $(a/p)$ is defined as
	\begin{equation*}
		(a/p)=\begin{cases}
			-1 &\text{ if $a$ is a quadratic nonresidue of $p$}\\
			1&\text{ if $a$ is a quadratic residue of $p$}.
		\end{cases}
	\end{equation*}
	One important property of the Legendre symbol is the following: If $a$ and $b$ are integers relatively prime to $p$ (where $p$ is an odd prime) and $a\equiv b(\text{mod }p)$, then $(a/p)=(b/p)$. For an expanded discussion of the Legendre symbol and its properties, see \cite{Burton}.
	
	\tab The concept of Legendre cordial labeling was first introduced in \cite{Andoyo}, where it was shown that the path graph $P_p$, cycle graph $C_p$, star graph $S_p$, tadpole graph $T_{p,p}$, kayak paddle graph $KP_{p,p,p}$, bistar graph $B_{p,p}$, and fan graph $F_{p+1}$ all admit a Legendre cordial labeling modulo $p$, where $p$ is an odd prime. Additionally, the characterization of the Legendre cordial labeling of complete graph $K_n$ has been obtained in \cite{Andoyo2}. However, no study has been conducted for graphs obtained from graph operations. Thus, this paper explores the Legendre cordial labeling of some graphs obtained from graph operations: join, corona, lexicographic product, cartesian product, tensor product, and strong product.
	
	\tab Let $f:V(G)\to \{1,2,...,n\}$ be a bijective function, where $G$ is a simple connected graph of order $n$, and let $p$ be an odd prime. Then $f$ is said to be a \textit{Legendre cordial labeling modulo $p$} if the induced function $f_p^*:E(G)\to \{0,1\}$, defined by
	\begin{equation*}
		f_p^*(uv)=\begin{cases}
			0 &\text{ if }([f(u)+f(v)]/p) = -1\text{ or }f(u)+f(v)\equiv0(\text{mod }p)\\
			1 &\text{ if }([f(u)+f(v)]/p) = 1,
		\end{cases}
	\end{equation*}
	satisfies the condition $|e_{f_p^*}(0)-e_{f_p^*}(1)|\leq 1$, where $e_{f_p^*}(i)$ denotes the number of edges with label $i$ ($i=0,1$). A graph that admits a Legendre cordial labeling modulo $p$ is called a \textit{Legendre cordial graph modulo $p$}. 
	
	\section{Basic Concepts}
	
	\begin{definition}
		A path graph $P_n$ of order $n$ is a graph with vertex set $V(P_n)=\{v_1,v_2,...,v_n\}$ and edge set $E(P_n)=\{v_1v_2,v_2v_3,...,v_{n-1}v_n\}$.
	\end{definition}

	\begin{definition}
		A complete graph $K_n$ of order $n$ is a graph for which each two distinct vertices of $K_n$ are adjacent.
	\end{definition}

	\begin{definition}
		A graph $G$ is said to be a bipartite graph if its vertex set $V(G)$ can be partitioned into two non-empty subsets $V_1$ and $V_2$ such that every edge of $G$ has one end in $V_1$ and one end in $V_2$. 
	\end{definition}
	
	\begin{definition}{\normalfont\cite{Alameri,Chartrand,Entero}}
		Let $G_1$ and $G_2$ be graphs. Note that the vertex set of the graphs defined in items (iii) to (vi) is $V(G_1)\times V(G_2)$.
		\begin{enumerate}
			\item [i.] The join graph $G_1+G_2$ is a graph with vertex set 
			\begin{equation*}
				V(G_1+G_2)=V(G_1)\cup V(G_2)
			\end{equation*}
			and edge set
			\begin{equation*}
				E(G_1+G_2)=E(G_1)\cup E(G_2)\cup \{uv:u\in V(G_1)\text{ and }v\in V(G_2)\}.
			\end{equation*}
			\item[ii.] If $G_1$ has order $n$, then the corona graph $G_1\circ G_2$ is a graph obtained by taking one copy of $G_1$ and $n$ copies of $G_2$, and then joining the $i$th vertex of $G_1$ to every vertex of the $i$th copy of $G_2$. 
			\item[iii.] The lexicographic product graph $G_1[G_2]$ is a graph with edge set
			\begin{equation*}
				E(G_1[G_2])=\{(v_1,u_1)(v_2,u_2): v_1v_2\in E(G_1)\text{, or }v_1=v_2 \text{ and }u_1u_2\in E(G_2)\}.
			\end{equation*}
			\item[iv.] The cartesian product graph $G_1\square G_2$ is a graph with edge set
			{\small
				\begin{equation*}
					E(G_1\square G_2)=\{(v_1,u_1)(v_2,u_2): v_1=v_2\text{ and }u_1u_2\in E(G_2), \text{ or } u_1=u_2\text{ and }v_1v_2\in E(G_1)\}.
			\end{equation*}}
			\item[v.] The tensor product graph $G_1\times G_2$ is a graph with edge set
			\begin{equation*}
				E(G_1\times G_2)=\{(v_1,u_1)(v_2,u_2): v_1v_2\in E(G_1)\text{ and }u_1u_2 \in E(G_2)\}.
			\end{equation*}
			\item[vi.] The strong product graph $G_1\boxtimes G_2$ is a graph with edge set
			\begin{equation*}
				E(G_1\boxtimes G_2)=E(G_1\square G_2)\cup E(G_1\times G_2).
			\end{equation*}
		\end{enumerate}
	\end{definition}
	
	\begin{theorem}{\normalfont\cite{Weichsel}}\label{thm1}
		Let $G_1$ and $G_2$ be connected graphs. If $G_1$ or $G_2$ has a cycle of odd order, then the tensor product graph $G_1\times G_2$ is a connected graph.
	\end{theorem}
	
	\begin{theorem}{\normalfont\cite{Burton}}\label{thm2}
		The odd prime $p$ has $\frac{p-1}{2}$ quadratic residues and $\frac{p-1}{2}$ quadratic nonresidues in the set $\{1,2,...,p-1\}$.
	\end{theorem}

	\begin{theorem}{\normalfont\cite{Burton,Rosen}}\label{thm3}
		Let $p$ be an odd prime. Then
		\begin{equation*}
			(2/p)=\begin{cases}
				-1&\text{ if }p\equiv\pm 3(\text{\normalfont{mod} }8)\\
				1&\text{ if }p\equiv\pm 1(\text{\normalfont{mod} }8).
			\end{cases}
		\end{equation*}
	\end{theorem}
	\newpage
	
	\section{Main Results}
	\tab In the following discussion, $p$ is denoted as an odd prime.
	\begin{theorem}\label{CorGoP}
		Let $G$ be a connected graph of order $n$, $n\geq 2$, and let $p\equiv\pm 3(\text{\normalfont{mod} }8)$. If $G$ has size $n$ or $n\pm 1$, then the corona graph $G\circ P_{p-1}$ is a Legendre cordial graph modulo $p$. 
	\end{theorem}
	
	\begin{proof}
		According to Theorem~\ref{thm3}, $(2/p)=-1$. Suppose that $V(G)=\{v_1,v_2,...,v_n\}$ and $E(G)=\{\varepsilon_1,\varepsilon_2,...,\varepsilon_q\}$, $q\in\{n-1,n,n+1\}$. Let $P_{p-1}^i$ denote the $i$th copy of $P_{p-1}$ with $V(P_{p-1}^i)=\{u_1^i,u_2^i,...,u_{p-1}^i\}$, where $u_j^iv_i$ is an edge of $G\circ P_{p-1}$ for $j=1,2,...,p-1$ and $i=1,2,...,n$. Let $f:V(G\circ P_{p-1})\to \{1,2,...,np\}$ be a bijective function defined by
		\begin{equation*}
			f(u_j^i)=\begin{cases}
				j+\frac{p+1}{2}+p(i-1)&\text{ for } j=1,2,...,\frac{p-1}{2}\\
				j-\frac{p-1}{2}+p(i-1)&\text{ for } j=\frac{p+1}{2},\frac{p+3}{2},...,p-1,
			\end{cases}
		\end{equation*}
		and
		\begin{equation*}
			f(v_i)=\frac{p+1}{2}+p(i-1)
		\end{equation*}
		for $i=1,2,...,n$. 
		
		For the edges of $P_{p-1}^i$:
		\begin{equation*}
			f(u_j^i)+f(u_{j+1}^i)\equiv\begin{cases}
				2(j+1)(\text{mod }p)&\text{ for }j=1,2,...,\frac{p-3}{2}\\
				[2(j+1)-p](\text{mod }p)&\text{ for }\frac{p-1}{2},\frac{p+1}{2},...,p-2
			\end{cases}
		\end{equation*}
		for $i=1,2,...,n$. Given this labeling, for each $i=1,2,...,n$, let
		\begin{equation*}
			\alpha_i=\bigcup_{j=1}^{p-2}\{\xi:f(u_j^i)+f(u_{j+1}^i)\equiv \xi (\text{mod }p),\text{ }0<\xi<p\}
		\end{equation*}
		and consequently,
		\begin{equation*}
			\alpha_i=\{1\}\cup \{3,4,...,p-1\}.
		\end{equation*}
		By Theorem~\ref{thm2}, since $(2/p)=-1$ and $2\notin \alpha_i$, it follows that $p$ has $\frac{p-1}{2}$ quadratic residues and $\frac{p-3}{2}$ quadratic nonresidues in $\alpha_i$, for $i=1,2,...,n$.
		
		For the edges of the form $u_j^iv_i$:
		\begin{equation*}
			f(u_j^i)+f(v_i)\equiv(j+1)(\text{mod }p), \text{ for }j=1,2,...,p-2
		\end{equation*}
		and
		\begin{equation*}
			f(u_{p-1}^i)+f(v_i)\equiv 0(\text{mod }p)
		\end{equation*}
		for $i=1,2,...,n$. In view of this labeling, for each $i=1,2,...,n$, define
		\begin{equation*}
			\beta_i=\bigcup_{j=1}^{p-2}\{\xi:f(u_j^i)+f(v_i)\equiv\xi(\text{mod }p), \text{ }0<\xi<p\}.
		\end{equation*}
		Then 
		\begin{equation*}
			\beta_i=\{2,3,...,p-1\}.
		\end{equation*}
		Since $(1/p)=1$ and $1\notin \beta_i$, it follows by Theorem~\ref{thm2} that $p$ has $\frac{p-3}{2}$ quadratic residues and $\frac{p-1}{2}$ quadratic nonresidues in $\beta_i$, for $i=1,2,...,n$. Moreover, observe that
		\begin{equation*}
			f_p^{*}(u_{p-1}^iv_i)=0
		\end{equation*}
		for $i=1,2,...,n$.
		
		Lastly, for the edges of $G$: Because $\varepsilon_k = v_av_b$, for some $a,b\in \{1,2,...,n\}$, $a\neq b$, then 
		\begin{equation*}
			f(v_a)+f(v_b)\equiv 1(\text{mod }p)
		\end{equation*}
		and it follows that
		\begin{equation*}
			f_p^*(\varepsilon_k) = 1
		\end{equation*}
		for $k=1,2,...,q$. 
		
		From the above results, we conclude that
		\begin{equation*}
			e_{f_p^*}(0)=n\left(\frac{p-3}{2}\right)+n\left(\frac{p-1}{2}\right)+n
		\end{equation*}
		and 
		\begin{equation*}
			e_{f_p^*}(1)=n\left(\frac{p-1}{2}\right)+n\left(\frac{p-3}{2}\right)+q.
		\end{equation*}
		Since $q\in \{n-1,n,n+1\}$, we obtain $|e_{f_p^*}(0)-e_{f_p^*}(1)|\leq 1$. Hence, $G\circ P_{p-1}$ is a Legendre cordial graph modulo $p$.
	\end{proof}

	\begin{theorem}
		The tensor product graph $K_p\times G$, where $G$ is a connected bipartite graph, is a Legendre cordial graph modulo $p$.
	\end{theorem}
	
	\begin{proof}
		Let $V(K_p)=\{v_1,v_2,...,v_p\}$ and let $m$ be the size of $G$. Since $G$ is a bipartite graph, the vertex set $V(G)$ can be partitioned into two non-empty subsets $V_1$ and $V_2$ with $V_k=\{u_1^k,u_2^k,...,u_{r_k}^k\}$ for $k=1,2$. Thus,
		\begin{equation*}
			E(K_p\times G)=\bigcup_{u_x^1u_y^2\in E(G)}\{(v_iu_x^1)(v_j,u_y^2):i,j=1,2,...,p, \text{ }i\neq j\}.
		\end{equation*}
		
		Now, let $f:V(K_p\times G)\to \{1,2,...,p(r_1+r_2)\}$ be a bijective function defined as follows: For $s_k=1,2,...,r_k$,
		\begin{equation*}
			f((v_t,u_{s_k}^k))=\begin{cases}
				t+p(s_1-1)&\text{ for }k=1\text{ and }t=1,2,...,p\\
				p-t+p(s_2+r_1-1)&\text{ for }k=2\text{ and }t=1,2,...,p-1\\
				p+p(s_2+r_1-1)&\text{ for }k=2\text{ and }t=p.
			\end{cases}
		\end{equation*}
		Consequently, for $j=1,2,...,p$,
		\begin{equation*}
			f((v_j,u_x^1))+f((v_i,u_y^2))\equiv\begin{cases}
				(j-i)(\text{mod }p)&\text{ for }i=1,2,...,j-1,\\
				(j+p-i)(\text{mod }p)&\text{ for }i=j+1,j+2,...,p-1,\\
				j(\text{mod }p)&\text{ for }i=p,\text{ }j\neq p
			\end{cases}
		\end{equation*}
		where $u_x^1u_y^2\in E(G)$. Given this labeling, for each $j=1,2,...,p$, define
		\begin{equation*}
			\alpha_{u_x^1u_y^2}^j=\bigcup_{\substack{i=1\\i\neq j}}^p\{\xi:f((v_j,u_x^1))+f((v_i,u_y^2))\equiv\xi(\text{mod }p),\text{ }0<\xi<p\}
		\end{equation*}
		where $u_x^1u_y^2\in E(G)$. It follows that 
		\begin{equation*}
			\alpha_{u_x^1u_y^2}^j=\{1,2,...,p-1\}
		\end{equation*}
		for all $j=1,2,...,p$ and $u_x^1u_y^2\in E(G)$. By Theorem~\ref{thm2}, we have
		\begin{equation*}
			e_{f_p^*}(0)=e_{f_p^*}(1)= mp\left(\frac{p-1}{2}\right)
		\end{equation*}
		and hence $|e_{f_p^*}(0)-e_{f_p^*}(1)|=0$. Therefore, $K_p\times G$ is a Legendre cordial graph modulo $p$.
	\end{proof}
	In the following results, let $G_1$ and $G_2$ be graphs of orders $q_1$ and $q_2$, respectively. For each $i=1,2$, let $g_i:V(G_i)\to \{1,2,...,q_i\}$ be a bijective function. Define the following subsets of $E(G_i)$:
	\begin{equation}\label{eqvarrho}
		\varrho_i=\{ab\in E(G_i):([g_i(a)+g_i(b)]/p)=1\}
	\end{equation}
	and
	\begin{equation}\label{eqeta}
		\eta_i=\{ab\in E(G_i):([g_i(a)+g_i(b)]/p)=-1\text{ or }g_i(a)+g_i(b)\equiv 0(\text{mod }p)\}
	\end{equation}
	for $i=1,2$.
	
	In the case where the range of $g_i$ $(i=1,2)$ consists of $k_ip$ elements, that is, $G_i$ has order $k_ip$, where $k_i$ is a positive integer, we define the partition of range into subsets
	\begin{equation*}
		\lambda_t^i=\{tp+1,tp+2,...,tp+p\}
	\end{equation*}
	for $t=0,1,...,k_i-1$. Correspondingly, for each $t=0,1,...,k_i-1$, define 
	\begin{equation*}
		D_t^i=\{a\in V(G_i):g_i(a)\in\lambda_t^i\}.
	\end{equation*}
	\begin{theorem}\label{join}
		Let $G_1$ and $G_2$ be graphs of orders $np$ and $m$, respectively. If
		\begin{equation}\label{eq0.1}
			|\varrho_1|+|\varrho_2|=|\eta_1|+|\eta_2|+nm\pm1
		\end{equation}
		or
		\begin{equation}\label{eq0.2}
			|\varrho_1|+|\varrho_2|=|\eta_1|+|\eta_2|+nm,
		\end{equation}
		then the join graph $G_1+G_2$ is a Legendre cordial graph modulo $p$.
	\end{theorem}
	
	\begin{proof}
		Let $V(G_1)=\{v_1,v_2,...,v_{np}\}$ and $V(G_2)=\{u_1,u_2,...,u_m\}$. Define the bijective function $f:V(G_1+G_2)\to\{1,2,...,np+m\}$ by
		\begin{equation*}
			f(v_i)=g_1(v_i),\text{ for }i=1,2,...,np
		\end{equation*}
		and
		\begin{equation*}
			f(u_j)=g_2(u_j)+np,\text{ for }j=1,2,...,m.
		\end{equation*}
		
		For the edges of $G_k$: If $ab\in E(G_k)$, then 
		\begin{equation*}
			f(a)+f(b)\equiv [g_k(a)+g_k(b)](\text{mod }p),
		\end{equation*}
		which means
		\begin{equation}\label{eq1}
			f_p^*(ab)=\begin{cases}
				0&\text{ if }ab\in \eta_k\\
				1&\text{ if }ab\in \varrho_k
			\end{cases}
		\end{equation}
		for each $k=1,2$.
		
		For the remaining edges:
		\begin{equation}\label{eq2}
			f(v_i)+f(u_j)\equiv [g_1(v_i)+g_2(u_j)](\text{mod }p)
		\end{equation}
		for $i=1,2,...,np$ and $j=1,2,...,m$. 
		
		Now, for $j=1,2,...,m$ and $t=0,1,...,n-1$, let 
		\begin{equation*}
			\alpha_t^j=\bigcup_{a\in D_t^1}\{\xi:f(a)+f(u_j)\equiv\xi(\text{mod }p), \text{ }0\leq\xi<p\}
		\end{equation*}
		and by (\ref{eq2}), it follows that 
		\begin{equation}\label{eq3}
			\alpha_t^j=\{0\}\cup\{1,2,...,p-1\}.
		\end{equation}
		In accordance with (\ref{eq1}) and (\ref{eq3}), and Theorem~\ref{thm2}, we have
		\begin{equation*}
			e_{f_p^*}(0)=|\eta_1|+|\eta_2|+nm\left(\frac{p-1}{2}\right)+nm
		\end{equation*}
		and
		\begin{equation*}
			e_{f_p^*}(1)=|\varrho_1|+|\varrho_2|+nm\left(\frac{p-1}{2}\right).
		\end{equation*}
		By (\ref{eq0.1}) and (\ref{eq0.2}), it is immediate that $|e_{f_p^*}(0)-e_{f_p^*}(1)|\leq 1$, and so $G_1+G_2$ is a Legendre cordial graph modulo $p$.
	\end{proof}
	
	\begin{theorem}\label{corona}
		Let $G_1$ be a connected graph of order $n$ and $G_2$ be a graph of order $mp$. If
		\begin{equation}\label{eq7.1}
			|\varrho_1|+n|\varrho_2|=|\eta_1|+n|\eta_2|+nm\pm1
		\end{equation}
		or
		\begin{equation}\label{eq7.2}
			|\varrho_1|+n|\varrho_2|=|\eta_1|+n|\eta_2|+nm,
		\end{equation}
		then the corona graph $G_1\circ G_2$ is a Legendre cordial graph modulo $p$.
	\end{theorem}
	
	\begin{proof}
		Let $V(G_1)=\{v_1,v_2,...,v_n\}$ and $V(G_2)=\{u_1,u_2,...,u_{mp}\}$. Suppose that $G_2^i$ is denoted as $i$th copy of $G_2$ with $V(G_2^i)=\{u_1^i,u_2^i,...,u_{mp}^i\}$ for which each vertex of $G_2^i$ is adjacent to $v_i$, for $i=1,2,...,n$. Now, let $f:V(G_1\circ G_2)\to\{1,2,...,n+nmp\}$ be a bijective function defined by
		\begin{equation*}
			f(u_j^i)=g_2(u_j^i)+mp(i-1),\text{ for }j=1,2,...,mp
		\end{equation*}
		and
		\begin{equation*}
			f(v_i)=g_1(v_i)+nmp
		\end{equation*}
		for $i=1,2,...,n$.
		
		For the edges of $G_1$: If $ab\in E(G_1)$, then
		\begin{equation*}
			f(a)+f(b)\equiv [g_1(a)+g_1(b)](\text{mod }p)
		\end{equation*}
		and consequently,
		\begin{equation}\label{eq4}
			f_p^*(ab)=\begin{cases}
				0&\text{ if }ab\in \eta_1\\
				1&\text{ if }ab\in \varrho_1.
			\end{cases}
		\end{equation}
		
		Similarly, for the edges of $G_2^i$: If $a^ib^i\in E(G_2^i)$, then
		\begin{equation*}
			f(a^i)+f(b^i)\equiv [g_2(a)+g_2(b)](\text{mod }p)
		\end{equation*}
		and as a result
		\begin{equation}\label{eq5}
			f_p^*(a^ib^i)=\begin{cases}
				0&\text{ if }ab\in \eta_2\\
				1&\text{ if }ab\in \varrho_2
			\end{cases}
		\end{equation}
		for $i=1,2,...,n$. 
		
		For the remaining edges:
		\begin{equation}\label{eq6}
			f(v_i)+f(u_j^i)\equiv [g_1(v_i)+g_2(u_j)](\text{mod }p)
		\end{equation}
		for $i=1,2,...,n$ and $j=1,2,...,mp$. 
		
		Now, for every $i=1,2,...,n$ and $t=0,1,...,m-1$, define
		\begin{equation*}
			\alpha_t^i=\bigcup_{a\in D_t^2}\{\xi:f(v_i)+f(a^i)\equiv\xi(\text{mod }p),\text{ }0\leq\xi<p\}
		\end{equation*}
		and by using (\ref{eq6}), we have
		\begin{equation}\label{eq7}
			\alpha_t^i=\{0\}\cup \{1,2,...,p-1\}.
		\end{equation}
		By combining (\ref{eq4}), (\ref{eq5}), and (\ref{eq7}), and applying Theorem~\ref{thm2}, we obtain
		\begin{equation*}
			e_{f_p^*}(0)=|\eta_1|+n|\eta_2|+nm\left(\frac{p-1}{2}\right)+nm
		\end{equation*}
		and
		\begin{equation*}
			e_{f_p^*}(1)=|\varrho_1|+n|\varrho_2|+nm\left(\frac{p-1}{2}\right).
		\end{equation*}
		Thus, (\ref{eq7.1}) and (\ref{eq7.2}) implies that $|e_{f_p^*}(0)-e_{f_p^*}(1)|\leq 1$. Hence, $G_1\circ G_2$ is a Legendre cordial graph modulo $p$.
	\end{proof}
	
	\begin{theorem}\label{lexico}
		Let $G_1$ be a connected graph of order $n$ and size $n$. Also, let $G_2$ be a graph of order $mp$. If
		\begin{equation*}
			|\varrho_2|=|\eta_2|+m^2p,
		\end{equation*}
		then the lexicographic product graph $G_1[G_2]$ is a Legendre cordial graph modulo $p$.
	\end{theorem}
	
	\begin{proof}
		Let $V(G_1)=\{v_1,v_2,...,v_n\}$ and $V(G_2)=\{u_1,u_2,...,u_{mp}\}$. So,
		\begin{equation*}
			E(G_1[G_2])=\bigcup_{xy\in E(G_1)}\{(x,u_i)(y,u_j):i,j=1,2,...,mp\}.
		\end{equation*}
		Suppose that $f:V(G_1[G_2])\to \{1,2,...,nmp\}$ is a bijective function defined by
		\begin{equation*}
			f((v_i,u_j))=g_2(u_j)+p(i-1)
		\end{equation*}
		for $j=1,2,...,mp$ and $i=1,2,...,n$.
		
		For the edges of the form $(v_i,a)(v_i,b)$ where $ab\in E(G_2)$: 
		\begin{equation*}
			f((v_i,a))+f((v_i,b))\equiv [g_2(a)+g_2(b)](\text{mod }p)
		\end{equation*}
		and so
		\begin{equation}\label{eq8}
			f_p^*((v_i,a)(v_i,b))=\begin{cases}
				0&\text{ if }ab\in\eta_2\\
				1&\text{ if }ab\in\varrho_2
			\end{cases}
		\end{equation}
		for $i=1,2,...,n$.
		
		For the rest of the edges: 
		\begin{equation}\label{eq9}
			f((x,u_j))+f((y,u_k))\equiv[g_2(u_j)+g_2(u_k)](\text{mod }p)
		\end{equation}
		for $j,k=1,2,...,mp$ and $xy \in E(G_1)$.
		
		Next, for each $j=1,2,...,mp$, $t=0,1,...,m-1$, and $xy\in E(G_1)$, define the set
		\begin{equation*}
			\alpha_{j,t}^{xy}=\bigcup_{a\in D_t^2}\{\xi:f((x,u_j))+f((y,a))\equiv \xi (\text{mod }p),\text{ }0\leq \xi <p\}
		\end{equation*}
		which, by (\ref{eq9}), simplifies to
		\begin{equation}\label{eq10}
			\alpha_{j,t}^{xy}=\{0\}\cup\{1,2,...,p-1\}.
		\end{equation}
		
		Therefore, using (\ref{eq8}) and (\ref{eq10}), along with Theorem~\ref{thm2}, we obtain
		\begin{equation*}
			e_{f_p^*}(0)=n|\eta_2|+nm^2p\left(\frac{p-1}{2}\right)+nm^2p
		\end{equation*}
		and
		\begin{equation*}
			e_{f_p^*}(1)=n|\varrho_2|+nm^2p\left(\frac{p-1}{2}\right).
		\end{equation*}
		Because $|\varrho_2|=|\eta_2|+m^2p$, we have $|e_{f_p^*}(0)-e_{f_p^*}(1)|= 0$, which implies that $G_1 [G_2]$ is a Legendre cordial graph modulo $p$.
	\end{proof}
	
	\begin{theorem}\label{cart}
		Let $G_1$ and $G_2$ be connected graphs of orders $mp$ and $n$, respectively, and suppose that $G_2$ has size $nk$. If
		\begin{equation*}
			|\varrho_1|=|\eta_1|+mk,
		\end{equation*}
		then the cartesian product graph $G_1\square G_2$ is a Legendre cordial graph modulo $p$.
	\end{theorem}
	
	\begin{proof}
		Let $V(G_2)=\{u_1,u_2,...,u_n\}$. Suppose that $f:V(G_1\square G_2)\to \{1,2,...,nmp\}$ is a bijective function defined by
		\begin{equation*}
			f((a,u_j))=g_1(a)+mp(j-1), \text{ for }j=1,2,...,n
		\end{equation*}
		where $a\in V(G_1)$. 
		
		For the edges of the form $(a,u_j)(b,u_j)$ where $ab\in E(G_1)$:
		\begin{equation*}
			f((a,u_j))+f((b,u_j))\equiv [g_1(a)+g_1(b)](\text{mod }p)
		\end{equation*}
		which leads to
		\begin{equation}\label{eq11}
			f_p^*((a,u_j)(b,u_j))=\begin{cases}
				0&\text{ if }ab\in \eta_1\\
				1&\text{ if }ab\in \varrho_1
			\end{cases}
		\end{equation}
		for $j=1,2,...,n$.
		
		For the remaining edges: 
		\begin{equation}\label{eq12}
			f((a,x))+f((a,y))\equiv 2g_1(a)(\text{mod }p)
		\end{equation}
		where $a\in V(G_1)$ and $xy\in E(G_2)$. 
		
		Now, for each $t=1,2,...,m-1$ and $xy\in E(G_2)$, let 
		\begin{equation*}
			\alpha_{t}^{xy}=\bigcup_{a\in D_t^1}\{\xi:f((a,x))+f((a,y))\equiv \xi(\text{mod }p),\text{ }0\leq \xi<p\}.
		\end{equation*}
		This expression simplifies to
		\begin{equation}\label{eq13}
			\alpha_{t}^{xy}=\{0\}\cup\{1,2,...,p-1\}
		\end{equation}
		as a consequence of (\ref{eq12}). 
		
		From (\ref{eq11}) and (\ref{eq13}), together with Theorem~\ref{thm2}, it follows that
		\begin{equation*}
			e_{f_p^*}(0)=n|\eta_1|+nmk\left(\frac{p-1}{2}\right)+nmk
		\end{equation*}
		and
		\begin{equation*}
			e_{f_p^*}(1)=n|\varrho_1|+nmk\left(\frac{p-1}{2}\right).
		\end{equation*}
		We observe that $|e_{f_p^*}(0)-e_{f_p^*}(1)|=0$ as a result of the equality $|\varrho_1|=|\eta_1|+mk$. Therefore, $G_1\square G_2$ is a Legendre cordial graph modulo $p$.
	\end{proof}
	
	\begin{theorem}\label{ten}
		Let $G_1$ and $G_2$ be connected graphs such that either $G_1$ or $G_2$ contains a cycle of odd order. Suppose that $G_1$ has order $np$. If
		\begin{equation*}
			|\varrho_1|=|\eta_1|,
		\end{equation*}
		then the tensor product graph $G_1\times G_2$ is a Legendre cordial graph modulo $p$.
	\end{theorem}
	
	\begin{proof}
		Note that $G_1\times G_2$ is a connected graph according to Theorem~\ref{thm1}. Suppose that $V(G_2)=\{u_1,u_2,...,u_m\}$ and define the bijective function $f:V(G_1\times  G_2)\to \{1,2,...,nmp\}$ by 
		\begin{equation*}
			f((a,u_j))=g_1(a)+mp(j-1)
		\end{equation*}
		for $j=1,2,...,m$.
		
		Now, if $ab\in E(G_1)$ and $xy\in E(G_2)$, then $(a,x)(b,y),(b,x)(a,x)\in E(G_1\times G_2)$, and as a result 
		\begin{equation*}
			f((a,x))+f((b,y))\equiv f((b,x))+f((a,y))\equiv [g_1(a)+g_1(b)](\text{mod }p).
		\end{equation*}
		As a consequence,
		\begin{equation*}
			f_p^*((a,x)(b,y))=f_p^*((b,x)(a,y))=\begin{cases}
				0&\text{ if }ab \in\eta_1\\ 
				1&\text{ if }ab \in\varrho_1.
			\end{cases}
		\end{equation*}
		If $G_2$ has size $q$, then we obtain
		\begin{equation*}
			e_{f_p^*}(0)=2|\eta_1|q\text{ and }e_{f_p^*}(1)=2|\varrho_1|q.
		\end{equation*}
		The condition $|\varrho_1|=|\eta_1|$ leads to $|e_{f_p^*}(0)-e_{f_p^*}(1)|=0$, which shows that $G_1\times G_2$ is a Legendre cordial graph modulo $p$.
	\end{proof}
	
	\begin{theorem}\label{strong}
		Let $G_1$ and $G_2$ be connected graphs of orders $3p$ and $n$, respectively, and suppose that $G_1$ has size $n-1$. If 
		\begin{equation*}
			|\varrho_1|=|\eta_1|+1,
		\end{equation*}
		then the strong product graph $G_1\boxtimes G_2$ is a Legendre cordial graph modulo $p$.
	\end{theorem}
	
	\begin{proof}
		The proof of this theorem follows by combining the arguments used in the proofs of Theorems~\ref{cart} and~\ref{ten}, noting that $E(G_1\boxtimes G_2)=E(G_1\square G_2)\cup E(G_1\times G_2)$. Specifically, by applying the method from the proof of Theorem~\ref{cart}, we find that the number of edges in $G_1\square G_2$ labeled $0$ and $1$ are 
		\begin{equation*}
			n|\eta_1|+3\left(\frac{p-1}{2}\right)(n-1)+3(n-1)\text{ and }n|\varrho_1|+3\left(\frac{p-1}{2}\right)(n-1),
		\end{equation*}
		respectively. Similarly, using the approach from the proof of Theorem~\ref{ten}, the number of edges in $G_1\times G_2$ labeled $0$ and $1$ are 
		\begin{equation*}
			2|\eta_1|(n-1)\text{ and }2|\varrho_1|(n-1),
		\end{equation*}
		respectively. Thus,
		\begin{equation*}
			e_{f_p^*}(0)=n|\eta_1|+3\left(\frac{p-1}{2}\right)(n-1)+3(n-1)+2|\eta_1|(n-1)
		\end{equation*}
		and
		\begin{equation*}
			e_{f_p^*}(1)=n|\varrho_1|+3\left(\frac{p-1}{2}\right)(n-1)+2|\varrho_1|(n-1).
		\end{equation*}
		Since $|\varrho_1|=|\eta_1|+1$, we have $|e_{f_p^*}(0)-e_{f_p^*}(1)|=1$. Thus, $G_1\boxtimes G_2$ is a Legendre cordial graph modulo $p$.
	\end{proof}

\end{document}